\documentclass[12pt,reqno]{amsart}

\usepackage{hyperref}

\usepackage[all]{xy}
\usepackage{amssymb}
\usepackage{amsmath}
\usepackage{amsthm}
\usepackage{amsfonts}

\xymatrixcolsep{2cm}

\numberwithin{equation}{section}

\allowdisplaybreaks[1]

\usepackage{etoolbox}

\makeatletter

\patchcmd{\thesubsection}{\arabic}{\arabic}{}{}

\patchcmd{\@seccntformat}{\@secnumfont}{%
\@secnumfont\expandafter\protect\csname format#1\endcsname}{}{}

\patchcmd{\@startsection}{\@afterindenttrue}{\@afterindentfalse}{}{}

\patchcmd{\subsection}{-.5em}{.3\linespacing}{}{}

\makeatother

\theoremstyle{plain}

\newtheorem{theorem}{Theorem}[section]

\newtheorem{proposition}[theorem]{Proposition}

\newtheorem{corollary}[theorem]{Corollary}

\theoremstyle{remark}

\newtheorem{remark}[theorem]{Remark}

\newcommand{\Ker}[1]{\ensuremath{\mathrm{Ker} (#1)}}

\newcommand{\cat}[1]{\ensuremath{\mathcal{#1}}}

\newcommand{\at}[2][]{\ensuremath{\mathrm{at}_{#1} (#2)}}

\newcommand{\END}[2][]{\ensuremath{\mathcal{E}\mathit{nd}_{#1} (#2)}}

\newcommand{\id}[1]{\ensuremath{\mathbf{1}_{#1}}}

\newcommand{\rk}[2][]{\ensuremath{\mathrm{rk}_{#1}(#2)}}

\newcommand{\C}{\ensuremath{\mathbb{C}}}

\newcommand{\struct}[1]{\ensuremath{\mathcal{O}_{#1}}}

\newcommand{\sstalk}[2]{\ensuremath{\struct{#1,\,#2}}}

\newcommand{\HOM}[3][]{%
\ensuremath{\mathcal{H}\mathit{om}_{#1}(#2,\, #3)}}

\newcommand{\DIFF}[4][]{\ensuremath{\mathcal{D}\mathit{iff}^{#1}_{#2}(#3,\,#4)}}

\newcommand{\Diff}[4][]{\ensuremath{\mathrm{Diff}^{#1}_{#2}(#3,\,#4)}}

\newcommand{\coh}[3]{\ensuremath{\mathrm{H}^{#1}(#2,\,#3)}}

\begin{document}

\title[On the relative opers in dimension one]{On the relative opers  in dimension one}

\author[A. Singh]{Anoop Singh}

\address{School of Mathematics, Tata Institute of Fundamental Research,
Homi Bhabha Road, Mumbai 400005, India}

\email{anoops@math.tifr.res.in}

\author[A. Upadhyay]{Abhitosh Upadhyay}

\address{School of Mathematics and Computer Science, Indian Institute of Technology, Goa, 403401, India}

\email{abhitosh@iitgoa.ac.in}

\subjclass[2010]{32C38, 14F10, 53C07}

\keywords{Relative oper, Relative differential operator, Relative jet bundle, Relative holomorphic connection.}

\begin{abstract}
We investigate the relative opers over the complex analytic  family of compact 
complex manifolds of relative dimension one. We introduce the notion of relative opers arising from the second fundamental form associated with a relative holomorphic connection. We also investigate the relative differential operators over the complex analytic family of compact complex manifolds whose symbol is the identity automorphism. We show that the set of equivalent relative opers arising from the second fundamental form 
is in bijective correspondence with the set of equivalent relative differential operators whose symbol is the identity automorphism.

\end{abstract}

\maketitle
\section{Introduction}\label{Intro}

The notion of opers were introduced by Beilinson, and Drinfeld in \cite{BD1}, \cite{BD}. In fact the germ of this notion was already introduced  in the work of Drinfeld and Sokolov in \cite{DS1},\cite{DS2}.
Since then there have been lot of study on this,
especially in the realm of mathematical physics.
There are  certain opers arising naturally as limits of Higgs bundles in the Hitchin components \cite{M}. 
Also, there is a profound applications of opers to the geometric Langlands program \cite{BF}, \cite{FG}.
Moreover, in \cite{B1}, Biswas introduced the notion of coupled connection over a compact Riemann surface, which is nothing but the $\text{GL}(n, \C)$-opers arising from the second fundamental form associated with a holomorphic connection.  

Motivated by these, in this article we introduce the notion of relative opers or 
relative $\text{GL}(n,\C)$-opers over the complex analytic family of compact complex manifolds, and generalise results from \cite{B1} in the relative context. For the theory of complex analytic family of compact complex manifolds see \cite{KS1}.
The generalisation  in the relative setup is important because the relative opers may correspond 
to the relative projective structures on the family of compact Riemann surfaces, as this correspondence holds for the absolute setup, that is, there is a  correspondence between $\text{PGL}(2, \C)$ opers and projective structures.
For the  relative projective 
structures see \cite[Section 7]{AB}.  In addition, it would  not be very difficult to show that the relative 
$\text{SL}(2)$-opers gives rise to relative projective 
structures as defined in  \cite[Section 7]{AB}.  Moreover, the space of  differential operators plays a crucial role while 
establishing the correspondence between the space of opers and  the space of  projective structures on a compact Riemann surface (see \cite[Section 6, Theorem 6.1]{B1}). Therefore, it is interesting to see firstly  the correspondence between relative opers and relative differential operators.

 A
Complex analytic family of compact complex manifold is equivalent to a surjective holomorphic proper submersion $\pi : X \to S$ between complex manifolds $X$ and $S$. Further we assume that the relative dimension is one, which is same as saying each fibre of $\pi$ is of dimension one, that is, a compact Riemann surface. Many results in this article are also true 
in the higher relative dimension. Therefore, we mention 
explicitly if the relative dimension is $\geq 1$ or exactly $1$.
We work in analytic category.


In section \ref{Rel_oper}, we define a {\it relative oper} over $\pi: X \to S$, the surjective holomorphic proper submersion of relative dimension $1$. A relative oper is also called  a $S$-oper. 
Let $E$ be a holomorphic vector bundle over $X$, and 
$\nabla$ a relative holomorphic connection on $E$.
Let $F$ be a subbundle of $E$. Then, we have the second 
fundamental form $\beta_{X/S}(F, \nabla)$ associated with the relative holomorphic connection $\nabla$ and subbundle $F$ (see subsection \ref{SFF}).
The second fundamental form $\beta_{X/S}(F, \nabla)$
gives a filtration \eqref{eq:1} of $E$ by subbundles of $E$ starting from $F$, see Proposition \ref{prop:1}. 
We also define the {\it relative oper associated with the second fundamental form} $\beta_{X/S}(F, \nabla)$.
There is a natural notion of equivalence of two relative opers, and we consider the set $\mathfrak{Op}^{SFF}_k(X/S)$  of equivalent relative opers associated with the second fundamental form, where $k$ is a positive integer and stands for the length of the filtration.

In section \ref{Rel_jet}, we recall the definitions of 
relative jet bundle, relative differential operator and 
relative holomorphic connection. We also state some results from \cite{BS} in the relative context necessary to prove theorems in the subsequent sections.

In section \ref{Rel_op_iso}, we construct the relative opers arising from the differential operators whose symbol is an identity automorphism. The first thing is to construct a relative holomorphic connection from the above mentioned differential operators, 
more precisely we prove the following (see Proposition \ref{prop:3}).
\begin{proposition}
\label{prop:0.1}
Let $\pi: X \to S$ be a surjective  holomorphic proper
submersion of complex manifolds of relative dimension $\geq 1$. For $k \geq 1$,
let $$P : E\rightarrow Sym^k\Omega^1_{X/S}\otimes E$$ be a relative differential operator of order $k$ with symbol $$\sigma_k(P) = \id{E} \in \coh{0}{X}{Sym^k \cat{T}_{X/S}\otimes \END[\struct{X}]{E}},$$  the identity isomorphism of $E$, where $\sigma_k$ is in \eqref{eq:6}.   Then, $P$ induces a relative holomorphic connection $\nabla_{P}$ on $(k - 1)$-th relative jet bundle $J^{k-1}_{X/S}(E)$ associated with $E$.
\end{proposition}

In view of above Proposition \ref{prop:0.1}, we conclude the following result (see Corollary \ref{cor:1}) which is restatement of the Theorem \ref{thm:1} 
\begin{theorem}
\label{thm:0.1}
Let $\pi: X \to S$ be a surjective  holomorphic proper
submersion of complex manifolds of relative dimension $
1$. For $k \geq 1$,
let $P : E\rightarrow Sym^k\Omega^1_{X/S}\otimes E$ be a 
relative differential operator of order $k$ with  symbol  
as identity morphism $\id{E}$ of $E$. Then the triple 
$(J^{k-1}_{X/S}(E), \nabla_P, \{\cat{K}_i\})$ is a relative oper associated to the second fundamental form
$\beta_{X/S}(Sym^{k-1}\Omega^1_{X/S}\otimes E, \nabla_P)$, where $\nabla_P$ is the relative holomorphic connection on $J^{k-1}_{X/S}(E)$ arising from $P$ in Proposition \ref{prop:0.1}.
\end{theorem}

By a triple
$(E, P, \sigma_k(P) = \id{E})$, we mean that relative  differential operators of order $k$ from 
$E$ to $Sym^k \Omega^1_{X/S} \otimes E$ whose symbol
is the identity automorphism $\id{E}$ of $E$.
Again there is a natural notion for the equivalence of such triples (see end of the section \ref{Rel_op_iso}).
Let $\mathfrak{Diff}_k(X/S)$ be the set of equivalent triples of the form $(E, P, \sigma_k(P) = \id{E})$.

Now, in view of Theorem \ref{thm:0.1}, we get a map (see \eqref{eq:17})
\begin{equation}
\label{eq:0.1}
\Upsilon : \mathfrak{Diff}_k(X/S) \longrightarrow \mathfrak{Op}^{SFF}_k(X/S)
\end{equation}
defined by sending 
$(E, P, \sigma_k(P) = \id{E})$ to $(J^{k-1}_{X/S}(E), \nabla_P, \{\cat{K}_i\})$.
We will show that $\Upsilon$ is a bijective map (see Theorem \ref{thm:3}).

In the last section \ref{Bij}, we show the bijective correspondence between $\mathfrak{Diff}_k(X/S)$ and  $\mathfrak{Op}^{SFF}_k(X/S)$ by constructing an inverse map of $\Upsilon$. In particular, we show the following
(see Theorem \ref{thm:2})

\begin{theorem}
 \label{thm:0.2}
Let $\pi: X \to S$ be a surjective  holomorphic proper
submersion of complex manifolds of relative dimension $
1$. Let $(E, \nabla, E^F_{\bullet})$ be the relative oper associated to the 
second fundamental form $\beta_{X/S}(F, \nabla)$.
Then, there exists a relative differential operator
\begin{equation}
\label{eq:0.18}
P_{\nabla} : \cat{Q} \longrightarrow Sym^k \Omega^1_{X/S} \otimes \cat{Q}.
\end{equation}  
of order $k$ such that $\sigma_k (P_\nabla) = \id{\cat{Q}}$, where $\cat{Q}= E/F_{k-1}$, and $F_{k-1}$
is the last subbundle in the filtration $E^F_{\bullet}$
in \eqref{eq:1}.
 \end{theorem}

\section{Relative opers}
\label{Rel_oper}
In this section we define the notion of relative oper 
following \cite{BD}.
Let $\pi : X \rightarrow S$ be a surjective holomorphic proper submersion of relative dimension $1$. 
Then the sheaf $\Omega^1_{X/S}$ of relative holomorphic $1$-forms is a locally free sheaf of rank $1$, equipped with a universal $S$-derivation
$$\text{d}_{X/S}: \struct{X} \longrightarrow \Omega^1_{X/S},$$
that is $\text{d}_{X/S}$ is a $\pi^{-1}\struct{S}$- linear map and satisfies the Leibniz rule.

A {\it relative oper} or {\it $S$-oper} is a triple $(E, \nabla, E_\bullet)$ where
\begin{enumerate}
\item $E$ is a holomorphic vector bundle over $X$.
\item $\nabla$ is a relative holomorphic connection on $E$, that is,
$$\nabla : E \rightarrow \Omega^1_{X/S}\otimes E $$
is a $\pi^{-1}\struct{S}$-linear map, which satisfies the Leibnitz identity
$$\nabla(f s) = f \nabla(s) + \text{d}_{X/S}(f) s,$$
where $f$ is a local section of $\struct{X}$ and $s$ is a local section of $E$.
\item $E_\bullet : \hspace{.2 cm} 0 = E_0\subset E_1\subset\cdot\cdot\cdot \subset E_{n-1}\subset E_n = E$ is a filtration by subbundles of $E$ called $S$-oper flag.
\end{enumerate}
These data have to satisfy the following conditions:
\begin{enumerate}
\item $\nabla (E_i)\subset E_{i+1}\otimes \Omega^1_{X/S}$ for $1\leq i\leq n-1$.
\item The induced maps
$$\frac{E_i}{E_{i-1}} \xrightarrow{\nabla} \frac{E_{i+1}}{E_i}\otimes \Omega^1_{X/S}$$
are isomorphism for $1\leq i\leq n-1$.
\end{enumerate}
Given an $S$-oper $(E, \nabla, E_\bullet)$, we denote  $$\cat{Q} = E/E_{n-1}.$$ We define respectively  the degree, type and length of a relative oper $(E, \nabla, E_\bullet)$ as follows
$$\deg (E, \nabla, E_\bullet) := \deg (E),$$
$$\text{type} (E, \nabla, E_\bullet) := \rk{E},$$
$$\text{length} (E, \nabla, E_\bullet) := n.$$

We say that two relative opers $(E, \nabla, E_{\bullet} )$ and $(E', \nabla', E'_{\bullet})$ 
of same length $n$ are 
equivalent if there exists a holomorphic isomorphism 
$$\alpha : E \longrightarrow E'$$
such that the diagram
\begin{equation}
 \label{cd:1}
 \xymatrix{ E\ar[d]^{\alpha} \ar[r]^{\nabla} & E\otimes\Omega^1_{X/S}\ar[d]^{\alpha\otimes\id{\Omega^1_{X/S}}}\\
 E' \ar[r]^{\nabla'}& E'\otimes\Omega^1_{X/S}\\}
 \end{equation}
commutes, and $\alpha$ preserves the filtration, that is, $\alpha (E_i) = E'_i$ for every $1 \leq i \leq n$.

Let $\mathfrak{Op}_n(X/S)$ denote the set of all equivalent relative opers over $X/S$ of length $n$.

\subsection{Second fundamental form (SFF) and relative oper}
\label{SFF}
Let $E\xrightarrow{\varpi} X\xrightarrow{\pi} S$ be a holomorphic vector bundle equipped with a relative holomorphic connection $\nabla$. Let $F$ be a subbundle of $E$. The {\bf second fundamental form} of $F$ with respect to relative holomorphic connection $\nabla$ on $E$ is the following composition
$$F\xrightarrow{\iota} E \xrightarrow{\nabla}\Omega^1_{X/S}\otimes E\xrightarrow{\id{\Omega^1_{X/S}}\otimes q} \Omega_{X/S}^1\otimes (E/F),$$
denoted by $$\beta_{X/S}(F, \nabla) = (\id{\Omega^1_{X/S}}\otimes q)\circ\nabla\circ \iota,$$ where $q : E\longrightarrow E/F$ is the natural projection. In view of the Leibnitz identity, the second fundamental form $\beta_{X/S}(F, \nabla)$ is an $\struct{X}$-linear map.
The following proposition is true for any relative dimension.
\begin{proposition}\label{proposition1}
The second fundamental form $\beta_{X/S}(F, \nabla)$ induces a filtration
\begin{equation}\label{eq:1}
E^{F}_\bullet : 0 := F_0 \subset F_1 := F\subsetneq F_2\subsetneq F_3\subsetneq\cdot\cdot\cdot F_{n-1}\subsetneq F_n\subseteq E,
\end{equation}
of $E$ by subbundles.
\end{proposition}
\begin{proof}
Since $\beta_{X/S}(F, \nabla)$ is an $\struct{X}$-linear map, we have
$$\beta_{X/S}(F, \nabla)\in H^0(X, \HOM[\struct{X}]{\cat{T}_{X/S}\otimes F}{E/F}).$$
Let $$P \subset\frac{E/F}{\beta_{X/S}(F, \nabla)(\cat{T}_{X/S}\otimes F)}$$ be the torsion part of the cokernel of $\beta_{X/S}(F, \nabla).$ We have the natural projection
$$E/F \longrightarrow \frac{E/F}{\beta_{X/S}(F, \nabla)(\cat{T}_{X/S}\otimes F)},$$
and the inverse image $P$ under the above projection is the unique subbundle of $E/F$ of minimal rank containing the image 
$\beta_{X/S}(F, \nabla)(\cat{T}_{X/S}\otimes F)$.

Denote this subbundle of $E/F$ by $F'$. The inverse image $$q^{-1}(F') =: F_2$$ is the subbundle of $E$ containing $F$.

Now, we replace $F$ by $F_2$, and repeat the above construction, that is, we get second fundamental form $\beta_{X/S}(F_2, \nabla)$ of $F_2$ with respect to $\nabla$, and above process give us a subbundle $F_3$ and so on. Since $E$ is a vector budle of finite rank, the iterated construction of filtration stabilizes.
\end{proof}
\begin{remark}\label{Rem:1}\mbox{We have following observation from above construction.}
\begin{enumerate}
\item
The relative holomorphic connection $\nabla$ maps $F_i$ to $F_{i+1}\otimes\Omega^1_{X/S}$ for every $1 \leq i \leq n-1$.

\item $F$ is preserved by $\nabla$ if and only if $n = 1$.

\item The last subbundle, i.e., $F_n$ is preserved by $\nabla$, if not, we get another subbundle $F_{n+1}$ of $E$.
\item The second fundamental forms for the subbundles $\{F_i\}$ in the filtration (\ref{eq:1})  of $E$ give a homomorphism of vector bundles
\begin{equation}\label{homomorphism}
\alpha_i : F_i/ F_{i-1}\longrightarrow \Omega^1_{X/S}\otimes(F_{i+1}/F_i)
\end{equation}
for each $i = 1, 2, \cdot\cdot\cdot, n-1$, that is,
$\alpha_1$ coincides with $\beta_{X/S}(F, \nabla)$,
$\alpha_2$ coincides with $\beta_{X/S}(F_2, \nabla)$
and so on.

\item The filtration (\ref{eq:1}) may stabilize to a proper subbundle of $E$.
\end{enumerate} 
\end{remark}

The triple $(E, \nabla, E^F_\bullet)$ is called the \textbf{relative oper associated with the second fundamental form} $\beta_{X/S}(F, \nabla)$ if the corresponding filtration in (\ref{eq:1}) has the property that $F_n = E$, that is, filtration does not stabilize to a proper subbundle of $E$ and the homomorphism
$$\alpha_i : F_i/F_{i-1}\rightarrow \Omega^1_{X/S}\otimes (F_{i+1}/F_i)$$
is an isomorphism for all $i = 1, \cdot\cdot\cdot, n-1.$ 

If $S$ is a single point, then $X$ is a compact Riemann surface and in that case 
the triple $(E, \nabla, F)$ is called coupled connection \cite{B1}.

We say that two relative opers $(E_1, \nabla_1, E^{F_1}_{1, \bullet})$ and $(E_2, \nabla_2, E^{F_2}_{2, \bullet} )$ associated to the second fundamental forms
$\beta_{X/S}(F_1, \nabla_1)$ and $\beta_{X/S}(F_2, \nabla_2)$ respectively, are equivalent if 
 there is a holomorphic isomorphism $\varphi : E_1\rightarrow E_2$ such that $\varphi(F_1) = F_2$ and the following diagram
\begin{equation}
 \label{cd:2}
 \xymatrix{ E_1\ar[d]^{\varphi} \ar[r]^{\nabla_1} & E_1\otimes\Omega^1_{X/S}\ar[d]^{\varphi\otimes\id{\Omega^1_{X/S}}}\\
 E_2 \ar[r]^{\nabla_2}& E_2\otimes\Omega^1_{X/S}\\}
 \end{equation}
commutes.

Since $\varphi$ is an isomorphism and it maps $F_1$ onto $F_2$, and above digram \eqref{cd:2} commutes,
we have 
$$\text{length}(E_1, \nabla_1, E^{F_1}_{1, \bullet})
= \text{length}(E_2, \nabla_2, E^{F_2}_{2, \bullet}).$$

Let $\mathfrak{Op}^{SFF}_n(X/S)$ be the set of all equivalent relative opers associated with some second fundamental form. 
Then $\mathfrak{Op}^{SFF}_n(X/S) \subset \mathfrak{Op}_n(X/S)$.
Our aim is to classify  all relative opers associated with some second fundamental form in terms of certain type of 
relative differential operators.

\section{Relative Jet bundles, relative differential operators and relative holomorphic connections}
\label{Rel_jet}
In this section, we recall the notion of relative jet bundles, relative differential operators and relative holomorphic connections on a holomorphic vector bundle.
We also state some results which we will use to show our main 
theorem.

\subsection{Relative Jet bundle}
Let $\pi\,:\, X \,\longrightarrow\, S$ be a surjective  proper submersion of complex manifolds with relative dimension $ l \geq 1$. Also, assume that dimension of 
$X$ is $m$ and dimension of $S$ is $n$. Then, $m-n = l$.
  Let $E \xrightarrow{\varpi} X\xrightarrow{\pi} S$ be a holomorphic vector bundle. 
We define a bundle associated to $E$, called the {\bf relative jet bundle} as follows. Consider the following 
$$
J^1_{X/S}(E) := E\oplus(E\otimes\Omega^1_{X/S})$$
as $\pi^{-1}\struct{S}$-module. We equip $J^1_{X/S}(E)$ with a right $\struct{X}$-module structure
$$
(s, \sigma)\cdot f : = (fs, f\sigma + s\otimes \text{d}_{X/S} f),$$
where $s$ is a local section of $E$, $f$ is a local section of $\struct{X}$ and $\sigma$ is a local section of $E\otimes\Omega^1_{X/S}$. We shall always consider $J^1_{X/S}(E)$ with this right $\struct{X}$-module structure, and call it {\bf first order relative jet bundle} (see \cite{FB}) . This first order relative jet bundle $J^1_{X/S}(E)$ fits into  the following short exact sequence
\begin{equation}
\label{eq:2}
0\longrightarrow E\otimes\Omega^1_{X/S}\longrightarrow J^1_{X/S}(E)\xrightarrow{p_E} E\longrightarrow 0,
\end{equation}
of $\struct{X}$-modules.
Note that the short exact sequence \eqref{eq:2} need not 
be  holomorphically splitting as an $\struct{X}$-modules. We will see that the holomorphic splitting of \eqref{eq:2} is equivalent to the fact that $E$ admits a 
relative holomorphic connection.

We now define higher order  relative jet bundle and describe some of its functorial property. Consider the second order relative jets
$$J^2_{X/S}(E) = J^1_{X/S}(E)\oplus (E\otimes Sym^2\Omega^1_{X/S}) = E\oplus (E\otimes\Omega^1_{X/S}) \oplus (E\otimes Sym^2\Omega^1_{X/S})$$
as $\pi^{-1}\struct{S}$-module, where $Sym^2\Omega^1_{X/S}$ denotes the second symmetric power of $\Omega^1_{X/S}$.

Note that the relative derivation
$d_{X/S} : \struct{X} \to \Omega^1_{X/S}$ induces naturally the quadratic differential $d^{(2)}_{X/S}\,:\, \struct{X} \,\longrightarrow\, Sym^2\Omega^1_{X/S}$.

Now, we express the quadratic differential $d^{(2)}_{X/S}$ in terms of local coordinates. Let $x \in X$ be a point and let
$(U,\, \phi \,=\, (z_1,\,\cdots,\,z_l,\,z_{l+1},\,\cdots,\,z_{l+n}))$ be a holomorphic chart
on $X$ around $x$. Then, $\{dz_i\,\mid\, 1\,\leq\, i \,\leq\, l\}$ is an $\struct{U}$-basis of
$\Omega^1_{X/S}|_U$. Since $\pi:X \to S$ is a holomorphic proper submersion of relative dimension $l \geq 1$, for any holomorphic function 
$f$ on $U$, we have 
\begin{equation}
\label{eq:loc}
d_{X/S}(f) = \sum_{i =1}^l \frac{\partial f}{\partial z_i} d z_{i}.
\end{equation}
Now, using the local basis for $Sym^2\Omega^1_{X/S}$ over the holomorphic chart $(U, \phi)$, we can express  the quadratic differential $d^{(2)}_{X/S}(f)$  in the local co-ordinates $(U, \phi)$ as follows
$$ d^{(2)}_{X/S}(f) = \frac{1}{2!}\sum_{i, j}\frac{\partial^2 f}{\partial z_i\partial z_j}d z_i\odot dz_j,$$
where $\odot$ denotes the symmetric product. 

The right $\struct{X}$-module structure on $J^2_{X/S}(E)$ is defined as follows
$$
(s, \sigma, \tau)\cdot f = (fs, f\sigma + s\otimes \text{d}_{X/S} f, f\tau + \sigma\otimes \text{d} _{X/S} f + s\otimes d^{(2)}_{X/S}f).$$
Here $\sigma\otimes d_{X/S}f$, we mean that the image of $\sigma\otimes d_{X/S}f\in E\otimes\Omega^1_{X/S}\otimes\Omega^1_{X/S}$ in $E\otimes Sym^2 \Omega^1_{X/S}$ under the symmetrization map
$$E\otimes\Omega^1_{X/S}\otimes\Omega^1_{X/S}\longrightarrow E\otimes Sym^2\Omega_X^1.$$ 
It is easy to verify that the right $\struct{X}$-module structure on $J^2_{X/S}(E)$ is independent of the local coordinate system.

Inductively we define $k$-th order relative jets for $k \geq 1$ as follows.
$$J^k_{X/S}(E) := J^{k-1}_{X/S}(E)\oplus (E\otimes Sym^k\Omega^1_{X/S})$$
as $\pi^{-1}\struct{S}$-module and $Sym^k\Omega^1_{X/S}$ denote the $k$-th symmetric powers of $\Omega^1_{X/S}$.

Let $d^{(j)}_{X/S} : \struct{X}\longrightarrow Sym^j\Omega^1_{X/S}$ be the $j$-th order differential induced from the relative derivation $d_{X/S}$. Then, $d^{(j)}_{X/S}(f)$ can be expressed in the local coordinates
$(U, \phi)$ considered above as follows
$$d^{(j)}_{X/S}(f) = \frac{1}{j!}\sum_{i_1, \cdot\cdot\cdot, i_j}\frac{\partial^j f}{\partial z_{i_1}\cdot\cdot\cdot\partial z_{i_j}}dz_{i_1}\odot\cdot\cdot\cdot\odot dz_{i_j}.$$
Let $(s_0, s_1, \cdot\cdot\cdot, s_k)$ be a section of $J^k_{X/S}(E)$ with $s_i$ are local section of $E\otimes Sym^i\Omega^1_{X/S}$ for every $i = 0, \ldots, k$. Then, for any $f$ a local section of $\struct{X}$, we set
$$(s_0, s_1, \cdot\cdot\cdot, s_k)\cdot f = (t_0, t_1, \cdot\cdot\cdot, t_k),$$
where $t_i$ is a local section of $E\otimes Sym^i\Omega^1_{X/S}$ given by the following expression
$$t_i = \sum_{j =0}^i s_j\otimes d_{X/S}^{i -j} f.$$
Also, the right $\struct{X}$-module structure on $J^k_{X/S}(E)$ is independent of the local coordinate system. See \cite{FB} for more details on higher order 
relative jet bundles and higher order differentials.

In view of the definition of higher order relative jets 
$J^k_{X/S}(E)$ associated with $E$, we get 
 an exact sequence 
\begin{equation}\label{eq:3}
0\longrightarrow E\otimes Sym^k\Omega^1_{X/S}\longrightarrow J^k_{X/S}(E)\xrightarrow{p^k_E} J^{k-1}_{X/S}(E)\longrightarrow 0
\end{equation}
of $\struct{X}$-modules, for every $k \geq 1$.
The short exact sequence \eqref{eq:3} in general does not holomorphically split as an $\struct{X}$-modules,
although it splits holomorphically as $\pi^{-1}\struct{S}$-modules.

Let $F$ be another holomorphic vector bundle over $X/S$,
and $\Phi : E \longrightarrow F$ a holomorphic homomorphism. Then
from the above definition of relative jet bundle, it is immediate that $ \Phi : E\longrightarrow F$ induces a homomorphism
$$ J^k_{X/S}(\Phi) : J^k_{X/S}(E)\longrightarrow J_{X/S}^k(F),$$
for each $k \geq 0$ and the corresponding diagram of homomorphisms
\begin{equation}
\label{cd:3}
\xymatrix{
J^{k+1}_{X/S}(E)\ar[d]\ar[r]^{J^{k+1}_{X/S}(\Phi)} & J^{k+1}_{X/S}(F) \ar[d] \\
J^k_{X/S}(E) \ar[r]^{J^{k}_{X/S}(\Phi)} & J^{k}_{X/S}(F)  
}
\end{equation}
is commutative, where the vertical arrows are natural projections defined in \eqref{eq:3}. 

For any integer $k \geq 0$, from the definition of relative jet bundle and injective homomorphism in 
the short exact sequence \eqref{eq:3}
$$E \otimes Sym^k \Omega^1_{X/S} \longrightarrow J^k_{X/S}(E),$$
there is a natural injective homomorphism of vector bundles
\begin{equation}\label{eq:4}
\theta : J^{k+1}_{X/S}(E)\longrightarrow J^1_{X/S}(J^k_{X/S}(E)).
\end{equation}

Note that for $k = 0$, $\theta$ is an isomorphism.
We will explicitly describe $\theta$ for $k =1$. We have the natural projection (see \eqref{eq:2}) $$ p_E : J^1_{X/S}(E)\longrightarrow E.$$ 
The above projection induces a morphism
$$ J^1_{X/S}(p_E) : J^1_{X/S}(J^1_{X/S}(E))\longrightarrow J^1_{X/S}(E).$$
Next, consider the equation (\ref{eq:2}) and replacing $E$ by $J^1_{X/S}(E)$, we get another map
$$ p_{J^1_{X/S}(E)} : J^1_{X/S}(J^1_{X/S}(E))\longrightarrow J^1_{X/S}(E).$$
Note that $J^1_{X/S}(p_E)$ and $ p_{J^1_{X/S}(E)}$ both projects to $E$ under the composition with the projection $p_E : J^1_{X/S}(E)\longrightarrow E$. Therefore,
$$J^1_{X/S}(p_E)- p_{J^1_{X/S}(E)} : J^1_{X/S}(J^1_{X/S}(E))\longrightarrow E\otimes\Omega^1_{X/S}\subseteq J^1_{X/S}(E).$$
Now consider $\theta$ defined in equation (\ref{eq:4}) for $k =1$, then 
$$\text{Im}(\theta) = \Ker{J^1_{X/S}(p_E)- p_{J^1_{X/S}(E)}}.$$
It should be noted that the diagram

\begin{equation}
\label{eq:cd2}
\xymatrix{
J^1_{X/S}(J^{k+1}_{X/S}(E))\ar[d]\ar[r] & J^{k+1}_{X/S}(E) \ar[d] \\
J^1_{X/S}(J^k_{X/S}(E))  & J^{k+1}_{X/S}(E)  \ar[l]_{\theta}
}
\end{equation}
does not commute (unless $E = 0$ or $k = 0$).

\subsection{Relative differential operators}
In this section, we follow \cite{GD} and \cite{R} to  recall the definition of finite order relative differential operators, and symbol map associated with it. 

Let $E$ and $F$ be two vector bundles over $X \xrightarrow{\pi} S$. Let $k \geq 0$ be any integer. A $k$-th order relative differential operator (or $S$-differential operator) is a $\pi^{-1}\struct{S}$-linear homomorphism $$P : E\rightarrow F$$ such that for any open subset $U\subset X$ and for any $f\in \struct{X}(U),$ the bracket
$$[P|_U, f] : E|_U\rightarrow F|_U$$
defined as
$$[P|_U, f]_V (s) = P_V(f|_V s) - f|_V\hspace{.02 cm} P_V(s)$$
is a relative differential operator of order $(k -1)$, for any open subset $V\subset U$, and for all $s\in E(V)$.
For the case $k = 0$, we define a relative differential operator to be an $\struct{X}$-linear map from $E$ to $F$. 

Let $\HOM[S]{E}{F}$ be the sheaf of $\pi^{-1}\struct{S}$-linear morphism from $E$ to $F$. Then $\HOM[S]{E}{F}$ has $\struct{X}$-bimodule structure defined as follows:

For every local sections $f$ of $\struct{X}$, and $P$ of $\HOM[S]{E}{F}$, the $S$-linear morphisms $fP$ and $Pf$ are respectively, given by 
$$f P(\alpha) = f (P(\alpha)) \hspace{.3 cm}\text{and}\hspace{.3 cm} Pf(\alpha) = P(f\alpha),$$
where $\alpha$ is a local section of $E$. The first operation gives the left and second gives the right 
$\struct{X}$-module structure on $\HOM[S]{E}{F}$.
Unless and otherwise stated we always use  left $\struct{X}$-module structure on $\HOM[S]{E}{F}$.

Let $\Diff[k]{S}{E}{F}$ denote the set of all $S$-differential operators from $E$ to $F$ of order $k$. For any open subset $U$ of $X$, the assignment
$$U\,\longmapsto\, \Diff[k]{S}{E|_U}{F|_U}$$
is the sheaf of $S$-differential operators over $X$ of order $k$. This sheaf is denoted by $\DIFF[k]{S}{E}{F}$ and this is an $\struct{X}$-subbimodule of $\HOM[S]{E}{F}$. We have following increasing chain of inclusions of subsheaves
of $\HOM[S]{E}{F}$
$$\HOM[\struct{X}]{E}{F} \subset \DIFF[1]{S}{E}{F}\subset \DIFF[2]{S}{E}{F}\subset \cdots \subset \HOM[S]{E}{F}.$$
From \cite[Proposition 4.2]{BS}, we have 
 the following symbol exact sequence,
\begin{equation}
\label{eq:5}
0 \rightarrow \HOM[\struct{X}]{E}{F} \xrightarrow{\iota} \DIFF[1]{S}{E}{F}\xrightarrow{\sigma_1} \cat{T}_{X/S}\otimes \HOM[\struct{X}]{E}{F}\rightarrow 0,
\end{equation}
where $\sigma_1$ is the symbol map.

The above symbol exact sequence also makes sense for the higher order differential operator (see \cite[Chapter 2, Definition 7.15, Definition 7.18]{R}) and can be given as follows
\begin{equation}
\label{eq:6}
0 \rightarrow \DIFF[k-1]{S}{E}{F} \rightarrow \DIFF[k]{S}{E}{F}\xrightarrow{\sigma_k}Sym^k \cat{T}_{X/S}\otimes \HOM[\struct{X}]{E}{F}\rightarrow 0,
\end{equation}
where $\sigma_k$ denotes the $k$-th order symbol map.

\begin{remark}
\label{rem:1}
For a morphism  $\pi: X \to S$ of complex analytic spaces or complex algebraic varieties, 
the theory of {\it relative principal parts of order $n$} denoted as $\cat{P}^{(n)}_{X/S}$ has been 
developed in \cite[p.n.14,  16.3]{GD}, \cite[section 2]{G1} and \cite[section 3]{R}.
We also have notion of  {\it relative principal parts of order $n$} associated with a holomorphic vector bundle 
$E$ over $X/S$, denotes as $\cat{P}^{(n)}_{X/S}(E)$ (see \cite[p.n. 37,  16.7]{GD}).

Since,
we are considering that the complex analytic spaces $X$
and $S$ are smooth, that is, they are complex manifolds and $\pi$ is a holomorphic surjective proper submersion,
we have an isomorphism of vector bundles (see \cite[Proposition 4.2]{R})
$$\cat{P}^{(k)}_{X/S}(E) \cong J^k_{X/S}(E)$$
for every $k \geq 0$.

\end{remark}

We describe relative differential operators as functors 
on the category of $\struct{X}$-modules.
Let $\struct{X}-\textbf{Mod}$ denote the category of 
$\struct{X}$-modules. Fix an $\struct{X}$-module 
$F \in \text{Ob}(\struct{X}-\textbf{Mod})$.
Define a functor
\begin{equation}
\label{eq:7}
\cat{F}^{k}_F : \struct{X}-\textbf{Mod} \longrightarrow
\struct{X}-\textbf{Mod}
\end{equation}
by 
\begin{equation}
\label{eq:8}
\cat{F}^{k}_F(E) = \DIFF[k]{S}{E}{F}.
\end{equation}
Then, $\cat{F}^{k}_F$ is a contravariant functor.

In view of above Remark \ref{rem:1} and \cite[p.n. 41, Proposition 16.8.4]{GD}, we have 

\begin{proposition}
\label{prop:1} Let $\pi: X \to S$ be a surjective proper
submersion of complex manifolds. Then,
for every $k \geq 0$,
the contravariant functor $\cat{F}^k_F$ is representable. More precisely, it is represented by the $k$-th order relative jet bundle, that is,
\begin{equation}
\label{eq:8.1}
\cat{F}^k_{F}(E) = \DIFF[k]{S}{E}{F} \cong \HOM[\struct{X}]{J^k_{X/S}(E)}{F}
\end{equation}
\end{proposition} 
In fact, applying $\HOM[\struct{X}]{-}{F}$ to the short exact sequence \eqref{eq:3}, we get the symbol exact sequence \eqref{eq:6}.

\subsection{Relative holomorphic connection}
Now, we describe the relationship among relative jet bundles, relative differential operators and relative holomorphic connections.
For details on  relative holomorphic connections  see \cite{BS}. Consider the short exact sequence (\ref{eq:5}), and take $E = F$,  we get
\begin{equation}\label{eq:9}
0\rightarrow\HOM[\struct{X}]{E}{E}\rightarrow \DIFF[1]{S}{E}{E}\xrightarrow{\sigma_1} \cat{T}_{X/S}\otimes\HOM[\struct{X}]{E}{E}\rightarrow 0.
\end{equation}
We denote $\HOM[\struct{X}]{E}{E}$ by $\END[\struct{X}]{E}$. The subbundle $$ \cat{A}t_{S}(E) = \sigma_1^{-1}(\cat{T}_{X/S}\otimes \id{E})\subset \DIFF[1]{S}{E}{E}$$ is known as relative Atiyah bundle. We get a short exact sequence
\begin{equation}\label{eq:10}
0\rightarrow \END[\struct{X}]{E}\rightarrow\cat{A}t_{S}(E)\xrightarrow{\sigma_1} \cat{T}_{X/S}\rightarrow 0,
\end{equation}
which is known as relative Atiyah sequence. 

Let $\at[S]{E} \in \coh{1}{X}{\Omega^1_{X/S}\END[\struct{X}]{E}}$ denote the extension class of 
the short exact sequence \eqref{eq:10}, called the relative Atiyah class. Then we have well established known facts.
\begin{proposition}
\label{prop:2}
Let $\pi: X \to S$ be a surjective proper
submersion of complex manifolds, and $E$ be a holomorphic vector bundle over $X$. 
\begin{enumerate}
\item \label{a} $E$ admits a relative holomorphic connection.
\item\label{b} The relative Atiyah sequence \eqref{eq:10} splits 
holomorphically.
\item \label{c}
The relative Atiyah class $\at[S]{E}$ vanishes.
\item \label{d}
The first order relative jet bundle sequence 
in \eqref{eq:2} splits holomorphically.
\end{enumerate} 
\end{proposition}
\begin{proof}
For the equivalence of \eqref{a}, \eqref{b}, \eqref{c}
see \cite[Proposition 4.3]{BS} and \cite[Corollary 4.4]{BS}. Next, equivalence \eqref{b} and \eqref{d} follows from \eqref{eq:8.1}.  
\end{proof}

A relative holomorphic connection $\nabla$
on $E$ is in fact a relative first order differential 
operator whose symbol is an identity morphism of $E$.
More precisely, since $\nabla$ satisfies Leibniz identity
\begin{equation}\label{eq:11}
\nabla (f s) = f \nabla (s) + d_{X/S}(f)\otimes s,
\end{equation}
where $f$ is a local section of $\struct{X}$, and $s$ is a local section of $E$. 
From (\ref{eq:11}), we have
$$[\nabla, f](s) = d_{X/S}(f)\otimes s,$$
where $[\nabla, f](s) = \nabla(fs) - f\nabla (s).$
Note that $\nabla$ is in fact first order relative differential operator whose symbol $\sigma_1(\nabla)$ is the identity automorphism of $E$, because 
$$\sigma_1(\nabla)(d_{X/S}f)(s) = [\nabla, f](s) = d_{X/S}(f)\otimes s.$$
Thus, $\nabla\in \coh{0}{X}{\DIFF[1]{S}{E}{E \otimes \Omega^1_{X/S}}}$ such that $\sigma_1(\nabla) = \id{E}.$

From (\ref{eq:8.1}), $\coh{0}{X}{\DIFF[1]{S}{E}{E \otimes \Omega^1_{X/S}}} \cong \coh{0}{X}{ \HOM[\struct{X}]{J^1_{X/S}(E)}{E\otimes\Omega^1_{X/S}}},$ therefore, we have an $\struct{X}$-linear map 
$$\widetilde{\nabla} :  J^1_{X/S}(E)\longrightarrow E\otimes \Omega^1_{X/S},$$
which gives an splitting of short exact sequence 
\eqref{eq:2}, because $\sigma_1(\nabla) = \id{E}$.

Thus, a relative holomorphic connection on $E$ is a
holomorphic  map  $$\widehat{\nabla} : E\longrightarrow J^1_{X/S}(E)$$ (as $\struct{X}$-module) such that the composition
$$E\xrightarrow{\widehat{\nabla}} J^1_{X/S}(E)\xrightarrow{p_E} E$$
is the identity morphism $\id{E}$.

\section{Relative opers arising from relative differential operators with symbol an isomorphism}
\label{Rel_op_iso}
 We investigate the relative differential operators from $E$ to $Sym^k \Omega^1_{X/S} \otimes E$ whose symbol is the identity automorphism $\id{E}$.
\begin{proposition}
\label{prop:3}
Let $\pi: X \to S$ be a surjective  holomorphic proper
submersion of complex manifolds of relative dimension $\geq 1$. For $k \geq 1$,
let $$P : E\rightarrow Sym^k\Omega^1_{X/S}\otimes E$$ be a relative differential operator of order $k$ with symbol $$\sigma_k(P) = \id{E} \in \coh{0}{X}{Sym^k \cat{T}_{X/S}\otimes \END[\struct{X}]{E}},$$  the identity isomorphism of $E$, where $\sigma_k$ is in \eqref{eq:6}.   Then, $P$ induces a relative holomorphic connection $\nabla_{P}$ on $(k - 1)$-th relative jet bundle $J^{k-1}_{X/S}(E)$ associated with $E$.
\end{proposition}
\begin{proof}
From \eqref{eq:8.1}, we have  
$$\DIFF[k]{S}{E}{Sym^k\Omega^1_{X/S}\otimes E} \cong \HOM[\struct{X}]{J^k_{X/S}(E)}{Sym^k\Omega^1_{X/S}\otimes E}.$$ Therefore, the differential operator $P$ gives a morphism
$$\phi_P : J^k_{X/S}(E)\longrightarrow Sym^k\Omega^1_{X/S}\otimes E.$$
Now, $\phi_P$
 is an splitting of the short exact sequence (\ref{eq:3}), because $\sigma_k(P) = \id{E}$.

Next, the splitting $\phi_P$ defines a morphism
\begin{equation}\label{eq:13}
\Psi_P : J^{k-1}_{X/S}(E)\longrightarrow J^k_{X/S}(E)
\end{equation}
of vector bundles whose composition with the projection
$p^k_E$ in (\ref{eq:3}) is the identity automorphism $\id{J^{k-1}_{X/S}(E)}$.

Consider the following commutative diagram of vector bundles
 \begin{equation}
 \label{cd:4}
 \xymatrix@C=2em{ 0 \ar[r] & Sym^k\Omega^1_{X/S}\otimes E \ar[d] \ar[r] & 
 J^k_{X/S}(E) \ar[d]^{\theta} \ar[r] & J^{k-1}_{X/S}(E)
 \ar[r] \ar@{=}[d] & 0 \\
 0 \ar[r] & \Omega^1_{X/S}\otimes J^{k-1}_{X/S}(E) \ar[r] & J^1_{X/S}(J^{k-1}_{X/S}(E)) \ar[r] & J^{k-1}_{X/S}(E) 
 \ar[r] & 0 \\}
 \end{equation}
where $\theta$ is defined in equation (\ref{eq:4}), 
the top exact sequence is the  relative jet bundle exact sequence in \eqref{eq:3} and the bottom jet bundle exact sequence is obtained from \eqref{eq:2} by putting $J^{k-1}_{X/S}(E)$ in place of $E$. The morphism $\Psi_P$ in equation (\ref{eq:13}) composed with $\theta$ gives a morphism
$$\theta \circ \Psi_P : J^{k-1}_{X/S}(E)\longrightarrow J^1_{X/S}(J^{k-1}_{X/S}(E))$$
of vector bundles which is nothing but the splitting of bottom short exact sequence in (\ref{cd:4}).
From Proposition \ref{prop:2} \eqref{d}, $J^{k-1}_{X/S}(E)$ admits  a relative holomorphic connection.

Moreover, 
let $$\chi_P : J^1_{X/S}(J^{k-1}_{X/S}(E))\longrightarrow \Omega^1_{X/S}\otimes J^{k-1}_{X/S}(E)$$ be the morphism of vector bundles obtained from the splitting of the bottom exact sequence in (\ref{cd:4}). Then from \eqref{eq:8.1}, $\chi_P$ corresponds to a first order differential operator 
\begin{equation}\label{eq:14}
\nabla_P\in \coh{0}{X} {\DIFF[1]{S}{J^{k-1}_{X/S}(E)}{\Omega^1_{X/S}\otimes J^{k-1}_{X/S}(E)}}
\end{equation}
such that $\sigma_1(\nabla_P)$ is the identity automorphism $\id{J^{k-1}_{X/S}(E)}$, which is nothing but the relative holomorphic connection in $J^{k-1}_{X/S}(E)$.
\end{proof}

Consider the following chain of projections of the vector bundle $J^{k-1}_{X/S}(E)$
\begin{equation}
J^{k-1}_{X/S}(E)\xrightarrow{{p}_{k-1}} J^{k-2}_{X/S}(E)\xrightarrow{{p}_{k-2}} J^{k-3}_{X/S}(E)\xrightarrow{{p}_{k-3}} \cdots \xrightarrow{{p}_{1}} J^0_{X/S}(E) = E\xrightarrow{{p}_0} 0.
\end{equation}
Let 
\begin{equation}
\label{eq:15}
\gamma_{k-1-i} : J^{k-1}_{X/S}(E)\rightarrow J^{k-1-i}_{X/S}(E)
\end{equation}
be the projection defined by the composition
$$\gamma_{k-1-i} = {p}_{k-1-i +1}\circ \cdots \circ {p}_{k-2}\circ{p}_{k-1},$$
for $i = 1, \ldots, k-1$.
We denote the kernel of $\gamma_{k-1-i}$ by $\mathcal{K}_i$, then we get following filtration of $J^{k-1}_{X/S}(E)$ 
\begin{equation}\label{eq:16}
0 = \mathcal{K}_0 \subset \mathcal{K}_1\subset \mathcal{K}_2\subset\cdot\cdot\cdot\subset\mathcal{K}_{k-1}\subset \cat{K}_k = J^{k-1}_{X/S}(E).
\end{equation}

Moreover, 
$E\otimes Sym^{k-1}\Omega^1_{X/S}$  is a subbundle of  $J^{k-1}_{X/S}(E),$ and from Proposition \ref{prop:3},  $\nabla_P$ is a 
  relative holomorphic connection on $J^{k-1}_{X/S}(E)$,
  then from Proposition \ref{prop:1}, we get a filtration (\ref{eq:1})  corresponding to the second fundamental form $\beta_{X/S}(E\otimes Sym^{k-1}\Omega^1_{X/S}, \nabla_P)$ of the subbundle $E\otimes Sym^{k-1}\Omega^1_{X/S}.$
\begin{theorem}\label{thm:1}
Let $\pi: X \to S$ be a surjective  holomorphic proper
submersion of complex manifolds of relative dimension $
1$. For $k \geq 1$,
let $P : E\rightarrow Sym^k\Omega^1_{X/S}\otimes E$ be a 
relative differential operator of order $k$ with  symbol  
as identity morphism $\id{E}$ of $E$. Then the 
filtration of $J^{k-1}_{X/S}(E)$ as defined in 
(\ref{eq:16}) coincides with the filtration (\ref{eq:1}) in 
Proposition \ref{prop:1}, after replacing $E$ by 
$J^{k-1}_{X/S}(E)$, $F$ by $E\otimes Sym^{k-1}
\Omega^1_{X/S}$ and $\nabla$ by $\nabla_P$.

Further, the homomorphism $\alpha_i$ defined in equation (\ref{homomorphism}) coincides with the identity automorphism of $Sym^{k-i}\Omega^1_{X/S}\otimes E$, where $i = 1, \dots,  k-1$.
\end{theorem}
\begin{proof}
Note that the terms of the filtration in \eqref{eq:16}
can be explicitly given as 
\begin{equation}
\label{eq:16.1}
\cat{K}_i =  \cat{K}_{i-1} \oplus Sym^{k-i}\Omega^1_{X/S}\otimes E,
\end{equation}
for $i = 1, \ldots, k$. 
Now, applying the same steps as in the proof of the 
 Proposition \ref{prop:1} 
for the  vector bundle $J^{k-1}_{X/S}(E)$, subbundle $F = Sym^{k-1}\Omega^1_{X/S}\otimes E$ and relative holomorphic  connection $\nabla_P$ on $J^{k-1}_{X/S}(E)$ 
we get the following terms of the filtration in \eqref{eq:1}
\begin{equation}\label{omegafiltration}
\left\{
\begin{array}{ll}
	F_1 = F = Sym^{k-1}\Omega^1_{X/S}\otimes E; \hspace{.3 cm}F_2 = F_1 \oplus (Sym^{k-2}\Omega^1_{X/S}\otimes E)\\
	\\
F_3 = F_2 \oplus (Sym^{k-3}\Omega^1_{X/S}\otimes E); \ldots F_i = F_{i-1} \oplus (Sym^{k-i}\Omega^1_{X/S}\otimes E)\\
\\
F_{k-1} = F_{k-2} \oplus (\Omega^1_{X/S}\otimes E); \hspace{.4 cm} F_k = F_{k-1} \oplus E = J^{k-1}_{X/S}(E).\\
\end{array}
\right.
\end{equation}
Thus, the two filtrations coincide and have same length $k$.

Next, we show that $\alpha_i = \beta_{X/S}(\cat{K}_i, \nabla_P)=\id{Sym^{k-i}\Omega^1_{X/S}\otimes E}$
for $i = 1, \ldots, k-1$. Since the relative dimension is 
$1$, $\Omega^1_{X/S}$ is a locally free sheaf of rank $1$, and hence  all its symmetric powers are locally free sheaf of rank $1$. Therefore, using the expression 
of $\cat{K}_i$ in \eqref{eq:16.1}, the successive quotients in the filtration \eqref{eq:16} has the same rank.
Now, we give another description of $\beta_{X/S}(\cat{K}_i, \nabla_P)$ as follows.
Consider the following commutative diagram
 \begin{equation}
 \label{cd:5}
 \xymatrix@C=2em{ &  &  & 0 \ar[d]  & 
     \\
 & & & \mathcal{K}_i\ar[d]^{\iota} & & \\
 & &  J^{k}_{X/S}(E)\ar[d]^{\varphi}  & J^{k-1}_{X/S}(E)
 \ar[l]_{\Psi_P} \ar[d]^{\gamma_{k-1-i}} & 
 \\
 0 \ar[r] & {Sym^{k-i}\Omega^1_{X/S}\otimes E} \ar[r]^{\iota} & J^{k-i}_{X/S}(E)\ar[r]^{p^{k-i}_E} \ar[d]& J^{k-i-1}_{X/S}(E) 
 \ar[r] \ar@{=}[d] & 0 \\
 & & J^{k-i-1}_{X/S}(E) \ar@{=}[r] & J^{k-i-1}_{X/S}(E) &  & \\}
 \end{equation}
where $\mathcal{K}_i = \Ker{\gamma_{k-1-i}}$ and
$\Psi_P$ is defined in \eqref{eq:13}.
Because of the commutativity of the above diagram
\eqref{cd:5}, we have
$$p^{k-i}_E \circ \varphi \circ \Psi_P \circ \iota = 0.$$ 
Thus, the morphism $\varphi \circ \Psi_P \circ \iota$
factors through $Sym^{k-i}\Omega^1_{X/S}\otimes E$, and hence we get a morphism
$$\mu_i : \mathcal{K}_i \rightarrow Sym^{k-i}\Omega^1_{X/S}\otimes E$$
which is nothing but the second fundamental form for the subbundle $\mathcal{K}_i$ of $J^{k-1}_{X/S}(E)$
with respect to $\nabla_P$, i.e., $\mu_i = \beta_{X/S}(\mathcal{K}_i, \nabla_P)$.
Also, note that $\mu_i (\cat{K}_{i-1}) = 0$, therefore we have 
$$\mu_i = \beta_{X/S}(\mathcal{K}_i, \nabla_P) : \frac{\mathcal{K}_i}{\cat{K}_{i-1}} \rightarrow Sym^{k-i}\Omega^1_{X/S}\otimes E.$$

Further consider the following commutative diagram
\begin{equation}
 \label{cd:6}
 \xymatrix@= 3em{ 0 \ar[r] & \mathcal{K}_i \ar[d]^\nu \ar[r] & 
 J^{k-1}_{X/S}(E) \ar[d] \ar[r]^{\gamma_{k-1-i}} & J^{k-1-i}_{X/S}(E)\ar@{=}[d]
 \ar[r] & 0 \\
 0 \ar[r] & Sym^{k-i}\Omega^1_{X/S}\otimes E \ar[r] & J^{k-i}_{X/S}(E) \ar[r]^{p^{k-i}_E} & J^{k-i-1}_{X/S}(E) 
 \ar[r] & 0 \\}
 \end{equation}
where $\nu$ is defined due to commutativity of the diagram. Note that $\nu$ coincides with $\mu_i$, and have property that it vanishes on the subbundle $\mathcal{K}_{i-1}\subset\mathcal{K}_i$.

Now the morphism in (\ref{cd:6}) induces a morphism
\begin{equation}\label{newmorphism}
\tilde{\nu} : \frac{\mathcal{K}_i}{\mathcal{K}_{i-1}} =
Sym^{k-i}\Omega^1_{X/S}\otimes E \longrightarrow Sym^{k-i}\Omega^1_{X/S}\otimes E
\end{equation}
which is an isomorphism. Thus,  the morphism $\alpha_i =  \beta_{X/S}(\mathcal{K}_i, \nabla_P)$ is the identity automorphism of $Sym^{k-1}\Omega^1_{X/S}\otimes E.$
This completes the proof of the theorem.
\end{proof}

From above Theorem \ref{thm:1}, we have 
\begin{corollary}
\label{cor:1}
Let $\pi: X \to S$ be a surjective  holomorphic proper
submersion of complex manifolds of relative dimension $
1$. For $k \geq 1$,
let $P : E\rightarrow Sym^k\Omega^1_{X/S}\otimes E$ be a 
relative differential operator of order $k$ with  symbol  
as identity morphism $\id{E}$ of $E$. Then the triple 
$(J^{k-1}_{X/S}(E), \nabla_P, \{\cat{K}_i\})$ is a relative oper associated to the second fundamental form
$\beta_{X/S}(Sym^{k-1}\Omega^1_{X/S}\otimes E, \nabla_P)$, where $\nabla_P$ is the relative holomorphic connection on $J^{k-1}_{X/S}(E)$ arising from $P$ in Proposition \ref{prop:3}.
\end{corollary}

We want to consider the set of all equivalent relative  differential operators of order $k$ from 
$E$ to $Sym^k \Omega^1_{X/S} \otimes E$ whose symbol
is the identity automorphism of $E$.
First we define the equivalence of two relative differential operators. 

Let us denote such differential operator by a triple
$(E, P, \sigma_k(P) = \id{E})$.

Let $(E_i, P_i, \sigma_k(P_i) = \id{E_i})$ be the two triples for $i = 1, 2$, that is, 
$E_1$ and $E_2$ are two holomorphic vector bundles over $X\xrightarrow{\pi} S$ and  $$P_i \in \coh{0}{X}{ \DIFF[k]{S}{E_i} {Sym^k\Omega^1_{X/S}\otimes E_i}}$$ for $i = 1, 2$, two relative differential operators of order 
$k$ with symbol $\sigma_k(P_i) = \id{E_i}$.

We say that $(E_1, P_1, \sigma_k(P_1) = \id{E_1})$ is
equivalent to $(E_2, P_2, \sigma_k(P_2) = \id{E_2})$ or
$P_1$ is equivalent to $P_2$ if there is a holomorphic isomorphism $T : E_1\rightarrow E_2$ such that the following diagram
\begin{equation}
 \label{cd:7}
 \xymatrix{ E_1\ar[d]^{T} \ar[r]^{P_1} & Sym^k\Omega^1_{X/S}\otimes E_1\ar[d]^{\id{Sym^k\Omega^1_{X/S}}\otimes T}\\
 E_2 \ar[r]^{P_2}& Sym^k\Omega^1_{X/S}\otimes E_2\\}
 \end{equation}
commutes.

Let $\mathfrak{Diff}_k(X/S)$ be the set of all equivalent triples $(E, P, \sigma_k(P) = \id{E})$.

Note that equivalent relative differential operators 
will produce equivalent relative opers, and therefore,
in view of Corollary \ref{cor:1}, we get a map 
\begin{equation}
\label{eq:17}
\Upsilon : \mathfrak{Diff}_k(X/S) \longrightarrow \mathfrak{Op}^{SFF}_k(X/S)
\end{equation}
defined by sending 
$(E, P, \sigma_k(P) = \id{E})$ to $(J^{k-1}_{X/S}(E), \nabla_P, \{\cat{K}_i\})$.
Our aim is to show that $\Upsilon$ is a bijective map.

\section{Bijective correspondence between $\mathfrak{Diff}_k(X/S)$ and  $\mathfrak{Op}^{SFF}_k(X/S)$}
\label{Bij}
In this section we show that the two sets $\mathfrak{Diff}_k(X/S)$ and  $\mathfrak{Op}^{SFF}_k(X/S)$ are in bijective correspondence, that is, the map 
$\Upsilon$ defined in \eqref{eq:17} is a bijective map.
Strategy is to construct a map from $\mathfrak{Op}^{SFF}_k(X/S)$ to $\mathfrak{Diff}_k(X/S)$ and then show that it is inverse of $\Upsilon$.
 
 \begin{theorem}
 \label{thm:2}
Let $\pi: X \to S$ be a surjective  holomorphic proper
submersion of complex manifolds of relative dimension $
1$. Let $(E, \nabla, E^F_{\bullet})$ be the relative oper associated to the 
second fundamental form $\beta_{X/S}(F, \nabla)$.
Then, there exists a relative differential operator
\begin{equation}
\label{eq:18}
P_{\nabla} : \cat{Q} \longrightarrow Sym^k \Omega^1_{X/S} \otimes \cat{Q}.
\end{equation}  
of order $k$ such that $\sigma_k (P_\nabla) = \id{\cat{Q}}$, where $\cat{Q}= E/F_{k-1}$, and $F_{k-1}$
is the last subbundle in the filtration $E^F_{\bullet}$
in \eqref{eq:1}.
 \end{theorem}

To show above theorem we will use  another description of 
the relative jet bundle given by using $k$-th infinitesimal neighbourhoods.

Let $\pi : X \longrightarrow S$ be as in the Theorem \ref{thm:2}.
Let  $\Delta \subset X\times_S X$  be the diagonal  as a closed complex submanifold of $X \times_S X$, and 
$\cat{I}$ the defining ideal sheaf of $\Delta$.
For each $k \geq 0$, the $k$-th infinitesimal neighbourhood of $\Delta$ in $X \times_S X$ is defined to be the complex analytic space 
$$\Delta^{(k)}_{X/S} := (\Delta, \struct{X \times_S X}/ \cat{I}^{k+1}).$$
We can view $\struct{\Delta^{(k)}_{X/S}}$ as a sheaf of 
$\struct{X}$-algebras in a natural way, that is, 
considering $\Delta^{(k)}_{X/S}$ as an analytic space over $X$ via the following morphism 
\begin{equation*}
\Delta^{(k)}_{X/S} \xrightarrow{\delta_k} X \times_S X \xrightarrow{\text{pr}_1} X,
\end{equation*}
where $\delta_k$ is arising from the fact that we have natural projection $\struct{X \times_S X} \longrightarrow \struct{X \times_S X}/ \cat{I}^{k+1}$.
As $\psi := \text{pr}_1 \circ \delta_k$  is the identity 
on the underlying spaces, $\psi^{-1} \struct{X} = \struct{X}$, $\psi_* \struct{\Delta^{(k)}_{X/S}} = \struct{\Delta^{(k)}_{X/S}}$ and $\psi^* \struct{X} = 
\struct{\Delta^{(k)}_{X/S}} $. We thus obtain a map of 
sheaves 
$$\text{pr}^*_1 : \struct{X} \longrightarrow \struct{\Delta^{(k)}_{X/S}}.$$
On the stalk level the morphism $\psi_x$ is the following composition
\begin{equation*}
\sstalk{X}{x} \xrightarrow{\text{pr}^{\sharp}_1} \sstalk{X \times_S X}{(x,x)} \xrightarrow{\delta^\sharp_k} \sstalk{X \times_S X}{(x,x)}/ \cat{I}^{k+1}_{(x,x)} = \sstalk{\Delta^{(k)}_{X/S}}{x}.
\end{equation*}

By the \textbf{relative jet} of order $k$ over $X/S$, denoted by $J^k_{X/S}$, we mean the structure sheaf of 
$\Delta^{(k)}_{X/S}$ viewed as a sheaf of $\struct{X}$-algebras via the map $\text{pr}^*_1 : \struct{X} \longrightarrow \struct{\Delta^{(k)}_{X/S}}$, that is,
\begin{equation}
\label{eq:19}
J^k_{X/S} := \struct{X \times_S X}/ \cat{I}^{k+1},
\end{equation}
together with its sheaf of $\struct{X}$-algebras structure described above.

Now for a vector bundle $E$ over $X$, we define
\begin{equation}
\label{eq:20}
J^k_{X/S}(E) := J^k_{X/S} \otimes_{\struct{X}} E
\end{equation}
as $\struct{X}$-module.

 Therefore, one can realize the section of $J^k_{X/S}(E)$ as  a section of $E$ restricting to the $k$-th order infinitesimal neighbourhood.
 
 \begin{proof}[\bf Proof of Theorem \ref{thm:2}]
 Let $(E, \nabla, E^F_{\bullet}) \in \mathfrak{Op}^{SFF}_k(X/S)$ be the relative oper associated to the 
second fundamental form $\beta_{X/S}(F, \nabla)$.
Note that $F_k = E$ and $F_{k-1}\subsetneq E$. Set $$\cat{Q} := \frac{E}{F_{k-1}}$$ to be the final quotient in the filtration  $E^F_{\bullet}$  in (\ref{eq:1}).
Let $p : E\rightarrow \cat{Q}$ be the natural projection. Take a point $x
\in X$ and a vector $v\in E_x$ in the fibre of $E$ over $x$. Let $U$ be an open subset of $X$ containing $x$, 
and $s_v : U \to E$ be a holomorphic section satisfying 
two conditions
\begin{enumerate}
\item $s_v (x) = v$
\item $\nabla (s_v) = 0$, i.e., $s_v$ is flat with respect to the connection $\nabla$ on $E$.
\end{enumerate}
Then $p(s_v)$ is a holomorphic section of $\cat{Q}$ defined around $x$.
Now, restricting $p(s_v)$ to the $k$-th order infinitesimal neighbourhood of $x$, we get an element $\eta_k(v)\in J^k_{X/S}(\cat{Q})_x$ in the fibre of $J^k_{X/S}(\cat{Q})$ over $x$.

Define a homomorphism for $j \geq 0$
\begin{equation}\label{homomorphismeta}
\eta_j : E\rightarrow J^j_{X/S}(\cat{Q})
\end{equation}
by sending any $v$ to the corresponding element $\eta_j(v)$ as above.


Now, we show that the 
homomorphism 
\begin{equation}\label{isomorphism}
\eta_{k-1} : E\rightarrow J^{k-1}_{X/S}(\cat{Q})
\end{equation}
is an isomorphism.
Note that rank $\rk{J^{k-1}_{X/S}(\cat{Q})} = k\hspace{.1 cm} \rk{\cat{Q}}.$ 
Since the given relative oper $(E, \nabla, E^F_{\bullet})$ has length $k$, it follows that $k \hspace{.1 cm}\rk{\cat{Q}} = \rk{E}.$ Thus, $\rk{E} = \rk{J^{k-1}_{X/S}(\cat{Q})}.$
Therefore, to show that $\eta_{k-1}$ is an isomorphism, it is enough to show that for any $x\in X$, $v\in E_x \setminus\{0\}$, $\eta_{k-1}(v) = 0$ implies $v = 0$.

By the  interval $[0, k-1]$, we mean the integer values from $0$ to $k-1$. First, we show the following assertion using induction on $i$:

If $\eta_i (v) = 0$ for some $i\in [0, k-1]$, then $v\in (F_{k-i-1})_x\subset E_x$ where $F_{k-i-1}$ is the subbundle of $E$ in filtration (\ref{eq:1}) and the homomorphism $\eta_i$ as defined in (\ref{homomorphismeta}).

For $i = 0$, we have  $\eta_0 : E \to \cat{Q}$. Therefore, $\eta_0 (v) = 0$ will imply that 
$v \in (F_{k-1})_x$. Thus, the assertion is true for 
$i = 0$.
 
Suppose that $\eta_i(v) = 0$ for $i = 0, 1$, and assertion is true for $i = 0$, then we show that assertion is true for $i = 1$.
Since $\eta_0(v) = 0$ and assertions is true for $i = 0$, $v \in (F_{k-1})_x$. 
Let $v_1 \in \frac{(F_{k-1})_x}{(F_{k-2})_x}$ be the image of $v$ by the natural projection $(F_{k-1})_x\rightarrow\frac{(F_{k-1})_x}{(F_{k-1})_x}$. The condition $\eta_1(v) = 0$ implies that image
$$\alpha_{k-1}(v_1) \in (\Omega^1_{X/S}\otimes F_k/F_{k-1})_x$$
is zero,
where $\alpha_{k-1}$ is in \eqref{eq:1}.
Since $\alpha_{k-1}$ is an isomorphism, we get 
$v_1 = 0$. Therefore, 
$v \in (F_{k-2})_x \subset E_x$. Thus, the assertion 
is true for $i = 1$.

Next, suppose that $\eta_i(v) = 0$, for  $0\leq i \leq (n+1)$ and assertion is true for all $i \leq n$. Using the similar steps as above, we show that assertion is true for $i = n+1$.
Since assertion is true for $i = n$, we get 
 $v\in (F_{k-n-1})_x \subset E_x.$

Let $v_n\in \frac{(F_{k-n-1})_x}{(F_{k-n-2})_x}$ be the image of $v$ by the natural projection 
$$(F_{k-n-1})_x\rightarrow \frac{(F_{k-n-1})_x}{(F_{k-n-2})_x}.$$
The condition $\eta_{n+1}(v) = 0$ implies that the image of $v_n$ under 
\begin{equation}
 \label{eq:21}
  \alpha_{k-n-1} : \frac{F_{k-n-1}}{F_{k-n-2}}  \longrightarrow \Omega^1_{X/S}\otimes\frac{F_{k-n}}{F_{k-n-1}}
 \end{equation}
 is zero, that is,  $\alpha_{k-n-1}(v_n) = 0$.

Since $\alpha_{k-n-1}$ is an isomorphism, this implies that $v_n = 0$. Therefore, $v \in (F_{k-n-2})_x\subset E_x$. In other words, the assertion is true for $i = n+1$.

Note that as $F_0 = 0$, setting $i = k-1$ in the above assertion we conclude that $$\eta_{k-1}(v) = 0 \Longrightarrow v = 0.$$
Thus, $\eta_{k-1}$ in \eqref{isomorphism} is an isomorphism.

Since $\eta_{k-1}$ is an isomorphism, 
 consider the homomorphism
$$\eta_k\circ \eta^{-1}_{k-1} : J^{k-1}_{X/S}(\cat{Q})\longrightarrow J^k_{X/S}(\cat{Q})$$
which is an splitting of the following jet bundle exact sequence 
\begin{equation*}
0\rightarrow Sym^k\Omega^1_{X/S}\otimes \cat{Q}\xrightarrow{\iota} J^k_{X/S}(\cat{Q})\xrightarrow{p^k_\cat{Q}} J^{k-1}_{X/S}(\cat{Q})\rightarrow 0.
\end{equation*}
The above splitting gives a homomorphism of vector bundles
\begin{equation}
\label{eq:22}
\widetilde{P}_\nabla : J^k_{X/S}(\cat{Q})\rightarrow Sym^k\Omega^1_{X/S}\otimes \cat{Q},
\end{equation}
such that $$\iota \circ \widetilde{P}_{\nabla} = \id{J^k_{X/S}(\cat{Q})}.$$
Since $\DIFF[k]{S}{\cat{Q}}{Sym^k\Omega^1_{X/S}\otimes \cat{Q}} \cong \HOM[\struct{X}]{J^k_{X/S}(\cat{Q})}{Sym^k\Omega^1_{X/S}\otimes \cat{Q}}$, we get a relative differential operator
\begin{equation}\label{eq:23}
P_\nabla : \cat{Q} \rightarrow Sym^k\Omega^1_{X/S}\otimes \cat{Q}
\end{equation}
of order $k$ such that $\sigma_k(P_\nabla) = \id{\cat{Q}}.$

This completes the proof of the theorem.
\end{proof}

\begin{remark}
\label{rm:2}
Under the assumption of  above Theorem \ref{thm:2}, we also get the following:
\begin{enumerate}
\item $\eta_{k-1}(F_i) = \mathcal{K}_i$ for each $i\in [0, k-1]$, where $F_i$'s are terms in the filtration (\ref{eq:1}).
\item There is an isomorphism
$$\overline{\eta}_i : \frac{E}{F_i}\rightarrow J^{k-1-i}_{X/S}(\cat{Q})$$
such that the following diagram
\begin{equation}
 \label{cd:8}
 \xymatrix{ E \ar[r]^{\eta_{k-1}}\ar[d] & J^{k-1}_{X/S}(\cat{Q}) \ar[d] \\
 E/F_i \ar[r]^{\overline{\eta}_i} & J^{k-1-i}_{X/S}(\cat{Q}) \\}
 \end{equation}
where $J^{k-1}_{X/S}(\cat{Q}) \rightarrow J^{k-1-i}_{X/S}(\cat{Q})$ is the projection.
\end{enumerate}
\end{remark}

It is easy to see that the equivalent relative opers will produce equivalent relative differential operators, so in view of Theorem \ref{thm:2}, we have a  map
\begin{equation}
 \label{eq:24}
 \Phi :  \mathfrak{Op}^{SFF}_k(X/S) \longrightarrow \mathfrak{Diff}_k(X/S) 
 \end{equation}
 defined by sending the triple $(E, \nabla, E^F_{\bullet})$ to the triple $(\cat{Q}, P_\nabla, \sigma_k(P_\nabla) = \id{\cat{Q}})$,
where $P_\nabla$ is constructed in (\ref{eq:23}) is a relative differential operator on $\cat{Q}$ of order $k$ such that $\sigma_k(P_\nabla) = \id{\cat{Q}}.$

\begin{theorem}
\label{thm:3}
Let $\pi : X\rightarrow S$ be a surjective holomorphic proper submersion of relative dimension $1$. Then the two maps $\Upsilon$ and $\Phi$ defined in (\ref{eq:17}) and (\ref{eq:24}) respectively,  are inverses of each other, that is 
\begin{equation}\label{eq:4.14}
\Phi \circ\Upsilon = \id{\mathfrak{Diff}_k(X/S)},
\end{equation}
\begin{equation}\label{eq:4.15}
\Upsilon \circ \Phi = \id{\mathfrak{Op}^{SFF}_k(X/S)}.
\end{equation}
\end{theorem}
\begin{proof}
To show (\ref{eq:4.14}), let $(E, P, \sigma_k(P) = \id{E})\in \mathfrak{Diff}_k(X/S)$. Then applying $\Upsilon$ on it, from Corollary
\ref{cor:1}, we get a relative oper  
$(J^{k-1}_{X/S}(E), \nabla_P, \{\cat{K}_i\})$  associated to the second fundamental form
$\beta_{X/S}(Sym^{k-1}\Omega^1_{X/S}\otimes E, \nabla_P)$, where $\nabla_P$ is the relative holomorphic connection on $J^{k-1}_{X/S}(E)$ arising from $P$ in Proposition \ref{prop:3}, and $\cat{K}_i$'s are the terms of the filtration in \eqref{eq:16}.

Now, applying $\Phi$ on $(J^{k-1}_{X/S}(E), \nabla_P, \{\cat{K}_i\})$  gives a relative differential operator
$P_{\nabla_P}$ on $\cat{Q} : = J^{k-1}_{X/S}(E)/\mathcal{K}_{k-1}$ such that $\sigma_k(P_{\nabla_P}) = \id{\cat{Q}}$.

Since $\cat{Q} := J^{k-1}_{X/S}(E)/\mathcal{K}_{k-1} \simeq E$, we get the $P_{\nabla_P}$ on $E$. Now using the same steps in Theorem \ref{thm:2}, we conclude that 
$P$ and $P_{\nabla_P}$ coincide.

Next to show (\ref{eq:4.15}), let $(E, \nabla, E^F_{\bullet}) \in \mathfrak{Op}^{SFF}_k(X/S)$. Now, 
applying $\Phi$ on it, from Theorem \ref{thm:2}, we get 
the triple $(\cat{Q}, P_\nabla, \sigma_k(P_\nabla) = \id{Q}) \in \mathfrak{Diff}_k(X/S)$,  where $\cat{Q} = \frac{E}{F_{k-1}}$ Apply $\Upsilon$ on 
the later triple, we get the triple  $(J^{k-1}_{X/S}(\cat{Q}), \nabla_{P_\nabla}, \cat{Q}^{\cat{K}_1}_{\bullet}) \in \mathfrak{Op}_k^{SFF}(X/S)$.
In the proof of the  Theorem \ref{thm:2}, from 
\eqref{isomorphism}, we have  $J^{k-1}_{X/S}(\cat{Q}) \simeq E$. Using the steps similar to the Theorem \ref{thm:1}, we get that $\nabla$ coincides with 
$\nabla_{P_\nabla}$ and filtration $E^F_\bullet$ coincides with the filtration $\cat{Q}^{\cat{K}_1}_\bullet$. This completes the proof.
\end{proof}

\section*{acknowledgements}
The authors would like to thank referees for their detailed and helpful comments.


\begin{thebibliography} {999}

\bibitem{BD1}
A.~Beilinson A., V.~Drinfeld, Quantization of Hitchin’s integrable system and Hecke eigensheaves, Preprint,
1991.
\bibitem{BD} A.~Beilinson, V.~Drinfeld, Opers,
arXiv:math/0501398

\bibitem{DS1}
V.~G.~Drinfeld and V.~V.~Sokolov, Equations of Korteweg-deVries type and simple
Lie algebras, {\it Soviet Mathematics Doklady}, vol. {\bf 23} (1981), No. {3}, p. 457–462.

\bibitem{DS2}
V.~G.~Drinfeld and V.~V.~Sokolov, Lie algebras and equations of Korteweg-deVries
type, {\it Journal of Soviet Mathematics}, vol. {\bf 30} (1985), p. 1975–2035.

\bibitem{B1} I.~Biswas, 
Coupled connections on a compact Riemann surface
{\it J. Math. Pures Appl.}, {\bf 82} (2003), pp. 1-42






\bibitem{BF}
D.~Ben-Zvi, E.~Frenkel, Spectral curves, opers and integrable systems, {\it Publ. Math. Inst. Hautes \'Etudes Sci.}
{\bf 94} (2001), 87–159.

\bibitem{AB}
D.~G.~L.~Allegretti, T. Bridgeland,  The monodromy of meromorphic projective structures. {\it Trans. Amer. Math. Soc.} {\bf 373} (2020), no. 9, 6321–6367.



\bibitem{BS} I.~Biswas and A.~Singh,  {On the relative connections,}
{\it Communications in Algebra}, 48 (2020), no. 4, 1452-1475.

\bibitem{FB}
F.~Bottacin,  {\it Atiyah classes of Lie algebroids.} Current trends in analysis and its applications, 375–393, Trends Math., Birkhäuser/Springer, Cham, 2015.





\bibitem{GD} A.~Grothendieck, J. Dieudonn\'e, EGA IV. \'Etude locale des
 sch\'emas et des morphismes de sch\'emas. IV,
{\it Inst. Hautes \'Etudes Sci. Publ. Math.} \textbf{32} (1967), 5--361.

\bibitem{G1}
 A.~Grothendieck. Techniques de construction en g\'eom\'etrie analytique VII. \'Etude locale des 
morphisme: \'el\'ements de calcul infinit\'esimal. S\'eminaire Henri Cartan, 13(14), 1960/61.

\bibitem{FG}
E.~Frenkel, D.~Gaitsgory, Local geometric Langlands correspondence and affine Kac–Moody algebras, in
Algebraic Geometry and Number Theory, {\it Progr. Math., Vol.} {\bf 253}, Birkh\"auser Boston, Boston, MA, 2006,
69–260


\bibitem{KS1} K.~Kodaira and D.~C.~Spencer, On deformations of complex analytic
structures, I, II, {\it Annals of Math.} \textbf{67}(1958), 328--466.



\bibitem{M} R.~Moosa, Jet spaces in complex analytic 
geometry: an exposition, 	arXiv:math/0405563

\bibitem{R}
S.~Ramanan, Global Calculus, Graduate Studies in Mathematics, vol. 65, American Mathematical Society, Providence, RI, 2005.

\bibitem{Y}
M.~Yang, Opers and Higgs bundles, Ph.D. Thesis, University of Illinois at Chicago,

\end{thebibliography}
\end{document}